\documentclass[letterpaper, 10 pt, conference]{ieeeconf}

\IEEEoverridecommandlockouts                              

\usepackage{amssymb,amsmath,amsfonts,mathtools}

\usepackage{algorithmic}
\usepackage{subfig}
\usepackage[colorlinks=true]{hyperref}
\usepackage{graphicx}
\usepackage{textcomp}
\usepackage{mathtools}
\usepackage[colorlinks=true]{hyperref}
\usepackage{algorithm}

\usepackage{calrsfs}

\newtheorem{theorem}{Theorem}[section]
\newtheorem{lemma}[theorem]{Lemma}
\newtheorem{definition}[theorem]{Definition}

\newtheorem{problem}[theorem]{Problem}

\newtheorem{remark}[theorem]{Remark}

\newtheorem{assumptions}[theorem]{Assumptions}

\DeclareMathAlphabet{\pazocal}{OMS}{zplm}{m}{n}

\renewcommand{\theequation}{\thesection.\arabic{equation}}
\renewcommand{\thefigure}{\thesection.\arabic{figure}}
\renewcommand{\thetable}{\thesection.\arabic{table}}

\newcommand{\sr}{\stackrel}

\newcommand{\rar}{\rightarrow}

\newcommand{\tri}{\sr{\triangle}{=}}

\newcommand{\bea}{\begin{eqnarray}}
\newcommand{\eea}{\end{eqnarray}}
\newcommand{\bes}{\begin{eqnarray*}}
\newcommand{\ees}{\end{eqnarray*}}
\newcommand{\bce}{\begin{center}}
\newcommand{\ece}{\end{center}}

\def\VR{\kern-\arraycolsep\strut\vrule &\kern-\arraycolsep}
\def\vr{\kern-\arraycolsep & \kern-\arraycolsep}

\newcommand{\ben}{\begin{enumerate}}
\newcommand{\een}{\end{enumerate}}

\newcommand{\bi}{\begin{itemize}}
\newcommand{\ei}{\end{itemize}}

\newcommand{\bp}{\begin{problem}}
\newcommand{\ep}{\end{problem}}
\newcommand{\hso}{\hspace{.1in}}
\newcommand{\hst}{\hspace{.2in}}

\newcommand{\bc}{\begin{center}}
\newcommand{\ec}{\end{center}}

\begin{document}
\title{\LARGE \bf  Comments and Corrections on  the DP Equations of Paper ``On Team Decision Problems With Nonclassical
Information Structures''
}

\author{Charalambos D. Charalambous$^1$, Umarbek Guvercin$^2$,  Seddik   Djouadi$^2$ 
\thanks{$^{1}$Charalambos D. Charalambous is  with the Faculty of Electrical and Computer Engineering, University of Cyprus, Nicosia 1678, Cyprus
{\tt\small chadcha@ucy.ac.cy}}
\thanks{$^{2}$Seddik Djouadi and  Umarbek Guvercin are  with the Faculty of Electrical Engineering and Computer Science, University of Tennessee, Knoxville, TN, 37996, USA 
{\tt\small \{mdjouadi,uguvercin\}@utk.edu}}%
\thanks{}%
}

%

\maketitle

\begin{abstract}
The 2023 paper   ``On Team Decision Problems With Nonclassical
Information Structures'' \cite{malikopoulosIEEEAC2023} presented information states and dynamic programming (DP) equations for  delayed sharing information patterns,  based on the concept of person-by-person (PbP) optimality of static team theory. In particular, \cite[Section~IV]  {malikopoulosIEEEAC2023}, Theorem~5 presents   recursions of the  information states  and Theorem~7, eqn(57), eqn(58),    presents   DP equations,  of each  team member. 

In this note  we show that the proof of  Theorem~5 and the DP eqn(57), eqn(58)  of Theorem~7 of  \cite{malikopoulosIEEEAC2023} are incorrect.  Consequently, Lemma~8,  Theorem~6, and  the full statement of Theorem~7  in  \cite{malikopoulosIEEEAC2023} are invalid, because their proofs  rely on erroneous   information states and the optimality of separated  strategies (i.e., functions of the information states). 

We further provide the  correct  DP equations  for  PbP optimality, thereby highlighting  
the subtleties and challenges inherent in the analysis of  delayed sharing information patterns.

\end{abstract}

\section{INTRODUCTION}
\label{sec:int}
Decentralized partially observable Markov decision problems (POMDPs) with delayed sharing patterns assigned to multiple controls or agents are introduced by Witsenhausen in the 1971  seminal paper \cite{witsenhausen1971}. Early studies  \cite{sandell-athans1974,yoshikawa1975,varaiya-walrand1978} derived a single dynamic programming  (DP) equation based on   a single cost-to-go  conditioned  on  the common or shared  information available to all agents 
by invoking  \cite[Assertion~8]{witsenhausen1971}, while  more recent studies  \cite{nayyar-mahajan-teneketzis2013,nayyar-teneketzis2019},  developed  several variations.  


The 2023 paper [1], presented  information states and    DP  equations  based on the concept of person-by-person (PbP) optimality of static team theory  \cite{marschak-radner1972,radner1962}, leading to multiple DP equations,  Theorems 5 and 7 of [1]. This note identifies gaps in the statements and proofs of these two theorems and proposes a new DP approach that corrects  Theorems 5 and 7 of [1]. The new DP approach is  an initial step toward a complete theory of PbP optimality, in a spirit reminiscent of the classical centralized DP framework for  POMDPs \cite{kumar-varayia:B1986,bertsekas2005}.


\subsection{Problem Formulation and Contributions}
We consider the problem formulation  in \cite{malikopoulosIEEEAC2023}, and  we follow the notation therein,  with minor variations when needed, to enhance the clarity of this presentation. 

The POMDP is described by the  state dynamics and observations of the system  given by 
 \begin{align}
&X_{t+1}=f_t(X_t, U_t^{(K)}, W_t), \;  t=0,1, \ldots,  T-1,   \label{NDM-3}\\
& Y_{t}^k   = h_t^k(X_t, Z_t^k),\;  t=0, \ldots, T, \;    k=1, \ldots,  K \label{NDM-4} 
\end{align}
 $X_t(\cdot): (\Omega, {\cal F}) \rar ({\mathbb  X}_t, {\cal X}_t)$  is the state, \\
 $Y_t^k(\cdot): (\Omega, {\cal F}) \rar ({\mathbb  Y}_t^k, {\cal Y}_t^k)$ is the $k$ observation,\\ $U_t^k(\cdot): (\Omega, {\cal F}) \rar ({\mathbb  U}_t, {\cal U}_t^k)$ is the $k$ control, $U_t^{(K)}=\{U_t^1, \ldots, U_t^K\}$, \\
$W_t(\cdot): (\Omega, {\cal F}) \rar ({\mathbb  W}_t, {\cal W}_t)$, $Z_t^k(\cdot): (\Omega, {\cal F}) \rar ({\mathbb  Z}^k_t, {\cal Z}_t^k)$, $\forall (t,k)$ are mutually independent  exogenous noises, which are independent of the initial state $X_0: (\Omega, {\cal F}) \rar ({\mathbb  X}_0, {\cal X}_0)$. \\
For each $(t, k)\in \{0,\ldots, T-1\} \times \{1, \ldots, K\}$ control input $U_t^k$ is assigned the $n-$step delayed sharing information pattern   denoted by $I_t^k(\cdot): (\Omega, {\cal F}) \rar ({\mathbb I}_t^k, {\cal  I}_t^k)$, as follows \cite{witsenhausen1971}:
  \begin{align}
 &I_t^k \tri \big\{Y_{0,t}^k, U_{0,t-1}^k\big\}\cup \Delta_t^{-k}=  \Delta_t \cup \Lambda_t^k, \label{ISG-1}\\
 &    \Delta_t^k \tri    \big\{Y_{0, t-n}^k, U_{0,t-n}^k \big\}, \;  \Lambda_t^k \tri     \big\{Y_{t-n+1, t}^k, U_{t-n+1, t-1}^k \big\}, \nonumber \\
  &\Delta_t\tri \Delta_t^k\cup \Delta_t^{-k}, \;  \Delta_t^{-k} =\cup_{j=1, j\neq k}^K \Delta_t^j,\; \Lambda_t^{-k} =\cup_{j=1, j\neq k}^K \Lambda_t^j \nonumber 
\end{align}
where $n\in \{1, \ldots, T\}$ is the time delay of sharing $\Delta_t^{-k}$.  Here, $\Delta_t$ is the common or shared information component available to all controls,  $U_t^j, j=1, \ldots, K$, while $\Lambda_t^k$  is the private information component available only to control $U_t^k$. 
We use the convention $B_{k,n}=\{B_k, B_{k+1}, \ldots, B_n\}, 0\leq k \leq n$   
 and $X_{k,n}=\{\emptyset\}, \forall k>n$.\\
For each $k\in \{1, \ldots, K\}$,  the admissible strategies of the $k$ control are  measurable functions $g^k(\cdot)$, 
\begin{align}
U_t^k=g_t^k(I_t^k)=g_t^k( \Delta_t, \Lambda_t^k),  \;  t=0, \ldots, T-1.    \label{type_a}
\end{align}
We denote the set of such   strategies  $g^k(\cdot)=\{g_1^k(\cdot), \ldots, g_{T-1}^k(\cdot)\}$ by   ${\bf  U}_{0,T-1}^k$ and  their $K-$tuple by   ${g}^{(K)}\tri \{g^1, \ldots, g^K\}\in  {\bf U}_{0,T-1}^{(K)}\tri  \times_{k=1}^K  {\bf  U}_{0,T-1}^{k}$.

The cost associated with any  $g^{(K)}\in {\bf   U}_{0,T-1}^{(K)}$ is 
\begin{align}
&J_T (g^{(K)}) \tri {\mathbb E}^{g^{(K)}} \Big\{ \sum_{t=0}^{T-1} c_t(X_t, g_t^{(K)})+c_T(X_T) \Big\} \label{i8-cost_a}
\end{align}  
where is   $c( \cdot)$ lower semi-continuous  and bounded from below (see \cite{bertsekas-shreve1978}  for general assumptions).

This note restricts attention to  Assumptions~\ref{ass-1} (as in   \cite{malikopoulosIEEEAC2023}).

\ \

\begin{assumptions}(Absolute Continuity\footnote{The existence of PDFs  are  standard assumptions   in most    POMDPs  \cite{kumar-varayia:B1986}.})
\label{ass-1}
The system model conditional probability distributions satisfy, 
\begin{align*}
& {\bf P}_{t+1}\Big\{X_{t+1} \in dx_{t+1}\big|X_t, U_{t}^{(K)}\Big\} 
=S_{t+1}(x_{t+1}|X_t,  U_{t}^{(K)})dx_{t+1},  \\
&{\bf P}_{t}\Big\{Y_t^k \in dy_t^k \big|X_t\Big\}=Q_{t}^k(y_t^k|X_t)dy_{t}^k, \; \forall (t,k)
\end{align*} 
i.e., $S_{t+1}(\cdot|x_t,  u_{t}^{(K)}),  Q_{t}^k(\cdot|x_t)$ are   probability density functions (PDFs)  (see  \cite{CDC:ECC2025,CDC-GU-SD:CDC2026-ArXiv,bertsekas-shreve1978}  for generalizations).  
\end{assumptions}

\ \

The material of \cite[Section~IV]  {malikopoulosIEEEAC2023} is focused on the concept of PbP optimality  of  static team theory \cite{marschak-radner1972,radner1962}.


\ \

 \begin{definition}(Person-by-Person Optimality) \\
 \label{def-pbp}
 A $K-$tuple of strategies   $g^{(K),o}\tri   \{g^{1,o},  \ldots, g^{K,o}\} \in {\bf   U}_{0,T-1}^{(K)}$
  is called  {\it person-by-person  (PbP) optimal} if it satisfies  the following inequalities, $\forall k \in \{1,2, \ldots, K\}$:
\begin{align}
& {J}_{T}(g^{k,o}, g^{-k,o}) \leq  {J}_{T}(g^{k}, g^{-k,o}), \;
 \forall g^k \in {\bf  U}_{0,T-1}^{k}, \label{pbp}\\
&J_{T}(g^{k,o}, g^{-k,o}) \tri  \inf_{  g^{k}\in {\bf  U}_{0,T-1}^k}  {J}_{T}(g^{k}, g^{-k,o})  \label{opt-p-k}
\end{align}
where $
 g^{-k,o} \tri \big\{g^{1,o}, \ldots,   g^{k-1,o}, g^{k+1,o}, \ldots, g^{K,o} \big\}$,  ${J}_{T}(g^{k,o}, g^{-k,o})$ is the optimal payoff of  strategy $g^k \in {\bf  U}_{0,T-1}^{k}$, when all other strategies are fixed to    $g^{-k,o}\in {\bf  U}_{0,T-1}^{-k} \tri \times_{j=1, j \neq k}^K{\bf  U}_{0,T-1}^{j}$. 
  \end{definition}

\ \

Paper   \cite[Section~IV]{malikopoulosIEEEAC2023} presented   DP equations based on PbP.
 i.e., for  the cost-to-go of each strategy 
 $g^k \in {\bf   U}_{0,T-1}^{k}$, based on Definition~\ref{def-pbp}.
%
  \cite[Theorems~5]{malikopoulosIEEEAC2023} derived   information state recursions for each strategy $g^{k}\in {\bf U}_{0,T-1}^k, k=1, \ldots, K$.  The information states are  then applied to obtain the DP equations of the  strategies   \cite[Theorems~7, eqn(57), eqn(58)]{malikopoulosIEEEAC2023}.   Additional properties of PbP optimal strategies are stated in  Lemma~7,  Lemma~8, Corollary~1, and  Theorem~6 in  \cite{malikopoulosIEEEAC2023}.

 Since Bellman's principle of optimality \cite{bertsekas2005,kumar-varayia:B1986} is central to this note,  in Definition~\ref{def:ctg-nm} we introduce the cost-to-go  of strategy $g^{k}\in {\bf  U}_{0,T-1}^{k}$  based on Definition~\ref{def-pbp}.

 \ \

\begin{definition}(PbP Cost-to-Go)\\
\label{def:ctg-nm}
For each $k \in \{1, \ldots, K\}$, consider  the payoff $J_{T} (g^{k}, g^{-k,o})$ of  strategy $g^k \in {\bf  U}_{0,T-1}^k$  for fixed   strategies  $g^{-k}=g^{-k,o}\in {\bf  U}_{0,T-1}^{-k}$. 
The cost-to-go ${\cal V}_{t}^{g^{k,o},g^{-k,o}}(\cdot): {\mathbb I}_t^{k}\rar (-\infty,\infty]$, over  $\{t,t+1, \ldots, T\}$  of   $g^k \in {\bf U}_{0,T-1}^k$,  when  optimal  $g^k=g^{k,o}\in {\bf  U}_{0,t-1}^{k}$ is  used over $\{0,1,\ldots, t-1\}$, conditioned  on any realization  $I_t^k=\{\Delta_t, \Lambda_t^k\}=\{\delta_t, \lambda_t^k\}$
  is  defined  by 
\begin{align}
 & {\cal V}_{t}^{g^{k,o},g^{-k,o}}(\delta_t, \lambda_t^k) \tri   \inf_{ g^{k}\in {\bf U}_{t,T-1}^k}  {J}_{t,T}^{\gamma^{k},\gamma^{-k,o}}(\delta_t, \lambda_t^k),   \;    \forall t 
  \label{opt-ctg-g}\\
&{J}_{t,T}^{g^{k},g^{-k,o}}(\delta_t, \lambda_t^k) 
 \tri   {\mathbb E}^{^{g^k, g^{-k,o}}} \Big\{\sum_{j=t}^{T-1} c_j(X_j,g_j^k,  g_j^{-k,o})  \nonumber \\
 &+ c_T(X_T) \Big|\Delta_t=\delta_t, \Lambda_t^k=\lambda_t^k   \Big\}, \;  t=0, \ldots, T, \hso \forall  k \label{opt-ctg-pbp-tc} 
\end{align}
$g_j^k\equiv g_j^k(\Delta_j, \Lambda_j^k)$, $g_j^{-k}\equiv \big\{g_j^1(\Delta_j, \Lambda_j^1), \ldots, g_j^{k-1}(\Delta_j, \Lambda_j^{k-1}),$ $g_j^{k+1}(\Delta_j, \Lambda_j^{k+1}),   \ldots, g_j^K(\Delta_j, \Lambda_j^K)\big\}$. For simplicity, we use the  notation $g_j^{-k}\equiv g_j^{-k}(\Delta_j, \Lambda_j^{-k})$. 
\end{definition}

\ \

Our main contributions  are  the following. 

\begin{enumerate}
\item We  provide proofs  that  the information states presented  in   Theorem 5 and the DP equations presented   in Theorem~7, eqn(57), eqn(58) of  \cite{malikopoulosIEEEAC2023} are erroneous.

\item   We present new DP equations  for PbP,  corresponding to   the cost-to-go  ${\cal V}_{t}^{g^{k,o},g^{-k,o}}(\delta_t, \lambda_t^k), k=1, \ldots, K$ (using the  recent  developments in  \cite{CDC:ECC2025,CDC-GU-SD:CDC2026-ArXiv}).

\end{enumerate}

We address item 1) by re-visiting the statement and proofs of Theorem 5 and  Theorem~7  in  \cite{malikopoulosIEEEAC2023}. Our findings include (a) inconsistencies and (b) gaps in the proofs of these basic Theorems. An  essential technical issue is that the proof of
Theorem 5 in \cite{malikopoulosIEEEAC2023} made use of a  conditional
independence condition (see  (\ref{d-6})), which does  not  hold. 

We address item 2) from first principles; the detailed derivations  are found in \cite{CDC:ECC2025,CDC-GU-SD:CDC2026-ArXiv}. 

The remainder of this note is organized as follows.

Section~\ref{sect:inc-mal} identifies inconsistencies and errors in the statements of Theorem 5, concerning the proposed information states, and Theorem 7, concerning the DP eqn(57) and eqn(58), in \cite{malikopoulosIEEEAC2023}.


Section~\ref{section:new} presents the new DP equations for PbP optimality  (based on  \cite{CDC:ECC2025,CDC-GU-SD:CDC2026-ArXiv}). 

Section~\ref{section:gaps}  examines the proof of Theorem 5 in \cite{malikopoulosIEEEAC2023} and highlights the  error in the derivation. 

Section~\ref{sect:inc-mal}, Section~\ref{section:new} and Section~\ref{section:gaps},  are  self-contained and may be read independently, and not necessarily in the order in which they appear.

\subsection{Inconsistencies in  the Information States and DP Equations of  \cite[Theorem~5, Theorem~7]{malikopoulosIEEEAC2023}}
\label{sect:inc-mal}
First, we reproduce\footnote{Text quoted from \cite{malikopoulosIEEEAC2023} is codified using  {\it  Italic characters}.} Theorem 5 and Theorem~7 in  \cite{malikopoulosIEEEAC2023}. 

\ \

{\bf \cite[Theorems~5] {malikopoulosIEEEAC2023}} {\it (Information State-Team Members):   For any strategy ${ g}^{(K)}=\{g^1, \ldots, g^K\} \in {\bf  U}_{0,T-1}^{(K)}$ of the team, the conditional probability density  ${\bf P}_t(X_t|\Delta_t, \Lambda_t^k)$ does not depend on the control strategy ${ g}^k$ of member $k$. It depends only on the strategy ${g}^{-k}=(g^1, \ldots, g^{k-1}, g^{k+1}, \ldots, g^K)$ of the other team members. It is an information state of the team member $k$, i.e., 
\begin{align}
\Pi_t^k(\Delta_t, \Lambda_t^k)(X_t)= {\bf P}_t^{{g}^{-k}}(X_t|\Delta_t, \Lambda_t^k), \hso \forall (t,k) \label{Mal_49-a}
\end{align}
 that can be evaluated from $\Delta_t$ and $\Lambda_t^k$. Moreover, there is a function $\theta^k$, which does not depend on the strategy $g^k$ of member $k$, such that
\begin{align}
& \Pi_{t+1}^k(\Delta_{t+1}, \Lambda_{t+1}^k)(X_{t+1})\nonumber \\
&=\theta_t^k\big(\Pi_{t}^k(\Delta_{t}, \Lambda_{t}^k)(X_t), \Delta_{t+1}, \Lambda_{t+1}^k\big) \label{Mal_49}
\end{align}
for all $t=0,1, \ldots, T-1$. }

\ \

We reproduce  \cite[Theorems~7] {malikopoulosIEEEAC2023}, where  (\ref{Mal_57}), (\ref{Mal_58}),  are  the DP equations based on PbP optimality \cite[eqn(57), eqn(58)] {malikopoulosIEEEAC2023}. 

\ \

{\bf \cite[Theorems~7] {malikopoulosIEEEAC2023}} {\it Let\footnote{$\Pi_t(\Delta_t, \Lambda_t^{(K)})-{\bf P}_t^{g^{(K)}}(X_t\big|\Delta_t, \Lambda_t^{(K)})$ is the  centralized information state and   ${\bf U}_{0,T-1}^{(K),sep}$ is the set of separated strategies $U_t^{k}=g_t^{sep}(\Pi_t(\Delta_t, \Lambda_t^{(K)}), k=1, \ldots,K$, both  introduced in  \cite[Definition~2, Section~IV.B, first paragraph] {malikopoulosIEEEAC2023}.} $V_t(\Pi_t(\Delta_t, \Lambda_t^{(K)})$
be functions defined recursively by 
\begin{align}
&V_T(\Pi_T(\Delta_T, \Lambda_T^{(K)}))= {\mathbb E}^{{ g}^{(K)}} \Big[ c_N(X_T) \Big| \Pi_T=\pi_T\Big],    \label{Mal_55}\\
& V_t(\Pi_t(\Delta_t, \Lambda_t^{(K)}))=\inf_{u_t^{(K)} \in \times_{k=1}^K{\mathbb U}_t^k} 
{\mathbb E}^{{ g}^{(K)}} \Big[ c_t(X_t, U_t^{(K)})  \nonumber \\
& + V_{t+1}(\Pi_{t+1}(\Delta_{t+1}, \Lambda_{t+1}^{(K)}) ) \Big| \Pi_t=\pi_t,  U_t^{(K)}=u_t^{(K)}\Big] \label{Mal_56}
\end{align}
where $\pi_t, u_t^{(K)}$ are realizations of $\Pi_t$ and $U_t^{(K)}$, respectively, and let $g^{(K),o} \in {\bf U}_{0,T-1}^{(K),sep}$  be the manager's optimal separated control strategy which achieves the infimum in (\ref{Mal_55})-(\ref{Mal_56}) for all $t=0, \ldots, T-1$. Let $g^{(K)}=(g^1, \ldots, g^{k-1}, g^k, g^{k+1}, \ldots, g^K)$ be the team's strategy, where $g^k=\{g_0^k \ldots, g_{T-1}^K\}$ is a separated control strategy of member $k \in {\cal K}\tri \{1, \ldots, K\}$ such that $U_t^k=g_t^k(\Pi_t^k(\Delta_t, \Lambda_t^k))$. 
 Let $V_t^k(\Pi_t^k(\Delta_t, \Lambda_t^k))$ be functions defined recursively by each team player $k \in \{1,2, \ldots, K\}$ after fixing ${ g}^{-k}=(g^1,\ldots, g^{k-1}, g^{k+1}, \ldots, g^k)$, by\footnote{Although, in  (\ref{Mal_57}), (\ref{Mal_58}), the author uses  ${\mathbb E}^{{g}^k}$ instead of   ${\mathbb E}^{g^{-k}}$ this is clearly a typographical error.}
\begin{align}
&V_T^k(\Pi_T^k(\Delta_T, \Lambda_T^k))= {\mathbb E}^{{\bf g}^{-k}} \Big[ c_N(X_T) \Big| \Pi_T^k=\pi_T^k\Big],    \label{Mal_57}\\
& V_t^k(\Pi_t^k(\Delta_t, \Lambda_t^k))=\inf_{u_t^k \in {\mathbb U}_t^k} 
{\mathbb E}^{{\bf g}^{-k}} \Big[ c_t(X_t, U_t^k, U_t^{-k})  \nonumber \\
& + V_{t+1}^k(\Pi_{t+1}^k(\Delta_{t+1}, \Lambda_{t+1}^k) ) \Big| \Pi_t^k=\pi_t^k, \Delta_t=\delta_t,\nonumber \\
&  \Lambda_t^k=\lambda_t^k, U_t^k=u_t^k\Big], \hso \forall t =0,\ldots, T-1,\forall k \label{Mal_58}
\end{align}
where $U^{-k}=(U_t^1, \ldots, U_t^{k-1}, U_t^{k+1}, \ldots, U_t^K)$, and $\pi_t^k, \delta_t$, $\lambda_t^k, u_t^k$ are realizations 
of  $\Pi_t^k, \Delta_t, \Lambda_t^k, U_t^k$, respectively.\\
Then the solution of the manager in (\ref{Mal_55}), (\ref{Mal_56}) is the same as the solution derived by each player $k$ in (\ref{Mal_57}), (\ref{Mal_58}) for all $t=0, 1, \ldots, T-1$.  }

\ \

We show below under 1)-3) that the statement of \cite[Theorems~5] {malikopoulosIEEEAC2023} contains  flaws  and inconsistencies.

\begin{enumerate}
\item First, we show that solutions to the recursion (\ref{Mal_49}) are  not Markovian (due to the dependence of $\theta_t^k(\Pi_t^k,\cdot, \cdot)$ explicitly on $(\Delta_{t+1}, \Lambda_{t+1}^k)$)    and consequently such solutions cannot be  candidates of  information states for strategies $g^k(\cdot)$, contradicting the statement,  \cite[Section~IV.B, first lines] {malikopoulosIEEEAC2023}:  {\it ``In view of Theorem~5, we show that the optimal separated control strategy $g^k = \{g_1^k, \ldots, g_{T-1}^K\}$, i.e., $U_t^k = g_t^k (\Pi_t^k(\Delta_t,\Lambda_t^k))$,
derived by each team member $k\in \{1,2, \ldots, K\}$  yields the same solution
as the one by the manager's optimal separated control strategies...''}. \\
To verify our claim that the proposed information state is non-Markovian, we examine the  conditional probability:
\begin{align}
&Prob\Big\{\Pi_{t+1}^k\Big| \Delta_{t}, \Lambda_{t}^k, U_t^k\Big\}\nonumber \\
&\sr{(a)}{=}Prob\Big\{\theta_t^k\big(\Pi_{t}^k, \Delta_{t+1}, \Lambda_{t+1}^k\big) \Big|  \Pi_t^k, \Delta_t, \Lambda_t^k, U_t^k\Big\}\nonumber \\
&\sr{(b)}{=}Prob\Big\{\theta_t^k\big(\Pi_{t}^k, \Delta_{t}, \Lambda_t^k, Y_{t+1}^k, U_t^k, Y_{t-T+1}^{-k}, U_{t-T+1}^{-k}\big)\nonumber \\
&\: \Big| \Pi_t^k,  \Delta_t, \Lambda_t^k, U_t^k \Big\}\label{inc_2} \\
&\neq Prob\Big\{\Pi_{t+1}^k\Big|\Pi_t^k,   U_t^k\Big\}\label{inc_1}
\end{align}
where $(a)$  follows from the recursion (\ref{Mal_49}) and the fact that $\Pi_t(\Delta_t, \Lambda_t^k)$ and $U_t^k=g_t^k(\Delta_t, \Lambda_t^k)$  are  measurable functions of $\{\Delta_t, \Lambda_t^k\}$, $(b)$ follows  from   the identity,  
\begin{align}
&\{\Delta_{t+1}, \Lambda_{t+1}^k\}\nonumber \\
&= \{\Delta_{t}, \Lambda_{t}^k, Y_{t+1}^k, U_t^k, Y_{t-T+1}^{-k}, U_{t-T+1}^{-k}\}
\end{align}
and (\ref{inc_1}) is obvious,  because the conditioning in (\ref{inc_2})  with respect to (w.r.t.) $ \{\Delta_{t}, \Lambda_{t}^k\}$ cannot be removed, i.e., the conditional probability distribution  depends on $ \{\Delta_{t}, \Lambda_{t}^k\}$ not only through $\Pi_t^k$ but also through the operator $\theta_t^k(\Pi_t^k, \cdot)$ dependence on the data $(\Delta_{t+1}, \Lambda_{t+1}^k)$. \\
 Consequently, since solutions to the recursion are  non-Markov,    the claim in  \cite[Secition~IV.B, first lines] {malikopoulosIEEEAC2023} that one can use separated control strategy, i.e., $U_t^k = g_t^k (\Pi_t^k(\Delta_t,\Lambda_t^k))$, is left without justification  (see classical centralized POMDPs \cite{kumar-varayia:B1986}). \\
 Moreover,  \cite[Theorem~7] {malikopoulosIEEEAC2023} is also affected, because  (\ref{Mal_58})  assumed $\Pi_t^k$ is Markov and separated strategies are optimal. That is, since $\Pi_t^k$ is not Markov,  the arguments of  $V_t^k(\cdot)$ and $V_{t+1}^k(\cdot)$ appearing    in (\ref{Mal_58}),    are not the ones indicated (the correct DP equations are given in Section~\ref{section:new}). 

\item For  the special case of    only a  single strategy  and observation, i.e., $K=1$,  the recursion of  \cite[Theorem~5]{malikopoulosIEEEAC2023} is expected to  correspond  to the nonlinear posterior PDF of centralized partially observable Markov decision problems (POMDPs) \cite{bertsekas2005,kumar-varayia:B1986}. This follows from the fact that, for $K=1$, there is only a single observation $Y_t^1, t=0, \ldots, T$,  and  $I_t^1=\{\Delta_t, \Lambda_t^1\}=\{Y_{0,t}^1, U_{0,t-1}^1\}, t=0, \ldots, T$. We show that  \cite[Theorem~5]{malikopoulosIEEEAC2023} does not reproduce the nonlinear posterior PDF of centralized POMDPs.  \\
For $K=1$, since there is 
 only one  posterior PDF, 
${\bf P}_t(X_t \big|Y_{0,t}^1, U_{0,t-1}^1)$,  this 
 satisfies  the well-known recursion of nonlinear filtering \cite{bertsekas2005,kumar-varayia:B1986} given by 
\begin{align}
&{\bf P}_{t+1}(X_{t+1} \big|Y_{0,t+1}^1, U_{0,t}^1) \nonumber \\
&=T^{filter}_{t}\big( {\bf P}_{t}(\cdot \big|Y_{0,t}^1, U_{0,t-1}^1),    Y_{t+1}^1, U_{t}^1\big), \hso \forall t.  \label{Mal_49-c}
\end{align}
Notice that the operator  $T^{filter}_{t}\big(\cdot)$ implies that  solutions ${\bf P}_{t}(X_{t} \big|Y_{0,t}^1, U_{0,t-1}^1), \forall t$ of (\ref{Mal_49-c})  are Markovian. \\
On the other hand, for $K=1$   recursion (\ref{Mal_49})  becomes  
\begin{align}
& \Pi_{t+1}^1(Y_{0,t+1}^1,U_{0,t}^1)(X_{t+1})\nonumber \\
&=\theta_t^1\big(\Pi_{t}^1(Y_{0,t}^1, U_{0,t-1}^1)(\cdot),  Y_{0,t+1}^1, U_{0,t}^1\big) \label{Mal_49-b}
\end{align}
Clearly, (\ref{Mal_49-b}) is  not consistent with  (\ref{Mal_49-c}), because   the two operators $T_t^{filter}(\cdot)$ and $\theta_t^1(\cdot)$ do not coincide. \\
The above is another indication  that  (\ref{Mal_49})  does not reproduce the standard nonlinear filtering recursion. 

%
%
\end{enumerate}
Since the DP equations in  \cite[Theorem~7, eqn(57), eqn(58)] {malikopoulosIEEEAC2023} incorrectly  assumed 
the  recursion (\ref{Mal_49}) is Markov,  then the statement of Theorems~7 is invalid. 

Below, we provided an alternative  independent illustration that  the DP equations  in  \cite[Theorem~7, eqn(57), eqn(58)] {malikopoulosIEEEAC2023}, i.e., (\ref{Mal_57}),   (\ref{Mal_58}), are   incorrect.

%
%

\begin{enumerate}
\setcounter{enumi}{2}
\item 
The right hand side of  (\ref{Mal_58})   involves the conditional expectation of the future cost
\begin{align}
&
{\mathbb E}^{ g^{-k}} \Big[ V_{t+1}^k(\Pi_{t+1}^k(\Delta_{t+1}, \Lambda_{t+1}^k) ) \Big| \Pi_t^k=\pi_t^k, \Delta_t=\delta_t,\nonumber \\
&  \Lambda_t^k=\lambda_t^k, U_t^k=u_t^k\Big] \label{Mal_58-d}\\
&\sr{(c)}{=}{\mathbb E}^{{\bf g}^{-k}} \Big[ V_{t+1}^k\big(\theta_t^k\big(\Pi_{t}^k, \Delta_{t}, \Lambda_t^k, Y_{t+1}^k, U_t^k,  Y_{t-T+1}^{-k}, \nonumber \\
& U_{t-T+1}^{-k}\big)\big) \Big| \Pi_t^k=\pi_t^k, \Delta_t=\delta_t,\Lambda_t^k=\lambda_t^k, U_t^k=u_t^k\Big] \label{Mal_58-e}\\
&\equiv G(\pi^k, \delta_t, \lambda_t^k, u_t^k), \hso \mbox{for some $G(\cdot)$} \label{Mal_58-ee}\
\end{align}
where $(c)$ is due to recursion (\ref{Mal_49}). 
On the other hand, the 
 left hand side  of (\ref{Mal_58}) is $V_t^k(\pi_t^k(\delta_t, \lambda_t^k))$, and this contradicts its right hand side (according to (\ref{Mal_58-ee})). 
\end{enumerate}

In view of the above,  Theorem~5 and  Theorem~7 in \cite{malikopoulosIEEEAC2023} are not valid. Moreover,  the additional statements in  Lemma~7,  Lemma~8, Corollary~1, Theorem~6 in  \cite{malikopoulosIEEEAC2023} are also left without any justification, because they rely on the validity of  Theorem~5 and  Theorem~7 in \cite{malikopoulosIEEEAC2023}.

%
%
 
%
%
%

\section{The New DP Equations and Information States for PbP Optimality}
\label{section:new}
We present the new DP equations and information states based on  PbP optimality (following 
 \cite{CDC-GU-SD:CDC2026-ArXiv,CDC:ECC2025}). 

In  
 Lemma~\ref{lemma:payoff-pbp}, we  express   the  payoff  of   strategy $g^k \in {\bf U}_{0,T-1}^k$, i.e.,  $J_{T} (g^{k}, g^{-k,o})$,   w.r.t.  the \^a posteriori  probability density function (PDF) of  nonlinear filtering of estimating the  extended state process $S_t^k \tri \{X_{t},   \Lambda_t^{-k}\}, t=0, \ldots, T$, which is unknown to $g^k(\cdot)$   from $I_t^k=\{\Delta_t, \Lambda_t^k\}, t=0, \ldots, T $. This  choice of extended state  $S_t^k$ follows from  an application of Bellman's principle of optimality. 

\ \

\begin{lemma}(Payoff of the Strategies)\\
\label{lemma:payoff-pbp}
For each $k$, define the  \^a posteriori  PDF of $\{X_t, \Lambda_t^{-k}\}$ conditional on  $I_t^k=\{\Delta_t^k, \Lambda_t^k\}$ by\footnote{Notation $\Xi_t^k[\Delta_t^k, \Lambda_t^k](x_t, \lambda_{t}^{-k})$ is often used  as a reminder that this is a measurable function of the data $I_t^k=\{\Delta_t^k, \Lambda_t^k\}$.} 
\begin{align}
 \Xi_t^k[\Delta_t^k, \Lambda_t^k](x_t, \lambda_{t}^{-k}) \tri 
 {\bf P}_t^{{g^k, g^{-k}}} (x_t, \lambda_{t}^{-k}\big|\Delta_t, \Lambda_t^k) \label{pPM}
\end{align} 
$ t=0, \ldots, T$.
The    payoff $J_{T} (g^{k}, g^{-k,o})$ of  strategy $g^k \in {\bf  U}_{0,T-1}^k$  for fixed   strategy  $g^{-k}=g^{-k,o}\in {\bf  U}_{0,T-1}^{-k}$ is expressed as, 
\begin{align}
&J_{T} (g^{k}, g^{-k,o}) ={\mathbb E}^{g^k, g^{-k,o}} \Big\{ \sum_{t=0}^{T-1} c_t( X_t,g_t^k,  g_t^{-k,o}) +c_T(X_T) \Big\}\nonumber  \\
&=   {\mathbb E}^{^{g^k, g^{-k,o}}} \Big\{ \sum_{t=0}^{T-1}  \int_{{\mathbb X}_t \times {\mathbb L}_{t}^{-k}}   c_t(x_t,g_t^k(\Delta_t^{k}, \Lambda_t^k),  g_t^{-k,o}(\Delta_t, \lambda_t^{-k}))\nonumber \\
&\hst .{\bf P}_t^{{g^k, g^{-k,o}}}(x_t, \lambda_{t}^{-k}\big|\Delta_t, \Lambda_t^k)dx_t d\lambda_t^{-k} \nonumber \\
&+\int_{{\mathbb X}_T}  c_T(x_T) {\bf P}_T^{{g^k, g^{-k,o}}}(x_T\big|\Delta_T, \Lambda_T^k)dx_T \Big\}, \label{tpayf-pbp}\\
&g_t^{-k}(\Delta_t, \lambda_t^{-k})\equiv \big\{g_t^1(\Delta_t, \lambda_t^1), \ldots, g_t^{k-1}(\Delta_t, \lambda_t^{k-1}), \nonumber \\
&g_t^{k+1}(\Delta_t, \lambda_t^{k+1}), \ldots, g_t^K(\Delta_t, \lambda_t^K)\big\}
\end{align}
where ${\bf P}_T^{{g^k, g^{-k,o}}}(x_T\big|\Delta_T, \Lambda_T^k)$ is the marginal of ${\bf P}_T^{{g^k, g^{-k,o}}}(x_T, \lambda_T^{-k}\big|\Delta_T, \Lambda_T^k)$.
\end{lemma}
\begin{proof}  (\ref{tpayf-pbp}) follows   by 
re-conditioning on $I_t^k$.
\end{proof}

\ \

\begin{remark}
\label{rem:1}
 There are multiple points to be made regarding Lemma~\ref{lemma:payoff-pbp}.

(I) For each $k$,  the payoff  (\ref{tpayf-pbp})  involves, not only  strategy  $g_t^k(\cdot)$ with argument $\{\Delta_t, \Lambda_t^k\}, t=1, \ldots, T-1$,  but also all other strategies $g_t^{m,o}(\cdot)$ whose arguments are $\{\Delta_t, \Lambda_t^m\}$,  for $t=1, \ldots, T-1$ and $ m\neq k$. Since  not only  the unobservable state  $X_t$, but also   the private information components $\Lambda_t^{-k}=\{\Lambda_t^m\}_{m=1, m\neq k}^K$,  are not available  to the  strategy $g_t^k$,  then  $S_t^k\tri \{X_t, \Lambda_t^{-k}\}$ needs to be estimated from the data $\{\Delta_t, \Lambda_t^k\}$.  Hence,  the representation  (\ref{tpayf-pbp}) that makes use of the posterior PDF ${\bf P}_t^{{g^k, g^{-k,o}}}(x_t, \lambda_t^{-k}\big|\Delta_t, \Lambda_t^k)$. 

(II)  Letting  $K=1$, representation  (\ref{tpayf-pbp})   reduces to the  representation of centralized  POMDPs \cite{bertsekas2005,kumar-varayia:B1986}.
 Indeed, if $K=1$ then $c_t(X_t, g_t^1, g_t^{-k})=\overline{c}(X_t, g_t^1(Y_{0,t}^1, U_{0,t-1}^1))$, i.e., $\{\Delta_t, \Lambda_t^1\}=\{Y_{0,t}^1, U_{0,t-1}^1\}$, 
and  in (\ref{tpayf-pbp}),   ${\bf P}_t^{{g^k, g^{-k,o}}}(x_t, \lambda_t^{-k}\big|\Delta_t, \Lambda_t^k)={\bf P}_t^{g^1}(x_t \big|Y_{0,t}^1, U_{0,t-1}^1)$, as expected. 

\end{remark}

\ \

The main point of  Remark~\ref{rem:1}, is that for each $k$, strategy $g^k(\cdot)$ is associated with the  {\it private posterior} PDF,  
$\Xi_t^k[\Delta_{t}, \Lambda_t^{k}] ={\bf P}_t^{{g^k, g^{-k}}}(x_t, \lambda_t^{-k}\big|\Delta_t, \Lambda_t^k), t=1, \ldots, T$, and this will play an important role in the characterization of PbP optimality via DP. 

In Theorem~\ref{thm:is-pbp},  we present   the   recursion of the PDF    $\big\{\Xi_t^k[\Delta_{t}, \Lambda_t^{k}]\big| t=0, \ldots, T\big\}$, and some of its properties.

\ \

\begin{theorem}(Recursion of Private   \^a Posteriori PDFs)\\
\label{thm:is-pbp}
Let 
$n \in \{1,2,\ldots, T\}$. For each $ k$,  any   
 $\{g^k,  g^{-k}\}\in {\bf  U}_{0,T-1}^k\times  {\bf U}_{0,T-1}^{-k}$,  and realization $I_t^k=\{\Delta_t, \Lambda_t^k\}=\{\delta_t, \lambda_t^k\}$, 
  the  private 
   \^a posteriori PDF of strategy $g^k$, 
   $\Xi_t^k[\Delta_t, \Lambda_t^k]=\xi_t^k[\delta_t, \lambda_t^k]\equiv {\bf P}_t^{{g^k, g^{-k}}} (x_t, \lambda_{t}^{-k}\big|\Delta_t=\delta_t, \Lambda_t^k=\lambda_t^k) $  satisfies   the recursion,
 \begin{align}    
&{\bf P}_{t+1}^{{g^k, g^{-k}}} (x_{t+1}, \lambda_{t+1}^{-k}\big|\delta_{t+1}, \lambda_{t+1}^k), \; \forall t,  \;  \forall k \label{apost_1}\  \\
&= {\bf T}_{t+1}^{k}\big(y_{t+1}^{k}, u_t^k, g_t^{-k}(\delta_t, \cdot),{\bf P}_{t}^{{g^k, g^{-k}}} (\cdot\big|\delta_{t}, \lambda_t^k)\big)(x_{t+1}, \lambda_{t+1}^{-k}) , \nonumber \\
& {\bf P}_0^{{g^k, g^{-k}}}(x_0, \lambda_0^{-k}|\delta_0, \lambda_0^k)= {\bf P}_0(x_0, y_0^{-k}|y_0^k)  
\end{align}
where $u_t^k=g_t^k(\delta_t, \lambda_t^k) \in {\mathbb U}_t^k$ and  the operator ${\bf T}_{t+1}^{k}\big(\cdot\big)(\cdot,\cdot)$  is defined  by  (non-zero if its denominator  is non-zero and zero otherwise),
\begin{align}
&{\bf T}_{t+1}^{k}\big(y_{t+1}^{k}, u_{t}^k,  g_t^{-k}(\delta_t, \cdot),\xi_{t}^{k} (\cdot)\big)(x_{t+1}, \lambda_{t+1}^{-k}) \tri \nonumber  \\
&\frac{\overline{\bf T}_{t+1}^k\big(y_{t+1}^{k}, u_{t}^k, g_t^{-k}(\delta_t, \cdot),\xi_{t}^{k} (\cdot)\big)(x_{t+1}, \lambda_{t+1}^{-k}) }{\int_{{\mathbb S}_{t+1}^k } \overline{\bf T}_{t+1}^k\big(y_{t+1}^{ k}, u_{t}^k, g_t^{-k}(\delta_t,\cdot),\xi_{t}^{k} (\cdot)\big)(x_{t+1},\lambda_{t+1}^{-k} )  dx_{t+1}d\lambda_{t+1}^{-k} }, \nonumber  \\
&{\mathbb S}_{t+1}^k \tri {\mathbb X}_{t+1}  \times {\mathbb L}_{ t+1}^{-k},  \hso  {\mathbb L}_{ t}^{-k} ={\mathbb Y}_{t-n+1, t}^{-k} \times {\mathbb U}_{t-n+1,t-1}^{-k}  \nonumber \\
&  \overline{\bf T}_{t+1}^k\big(y_{t+1}^{k}, u_{t}^k, g_t^{-k}(\delta_t, \cdot),\xi_{t}^{k} (\cdot)\big)(x_{t+1},\lambda_{t+1}^{-k}) \nonumber\\\
&\tri   Q_{t+1}^{k}(y_{t+1}^{k}|x_{t+1})    \prod_{j=1, j\neq k}^K Q_{t+1}^{j}(y_{t+1}^{j}|x_{t+1}) \nonumber \\
&.  \int_{{\mathbb X}_t}  S_{t+1}(x_{t+1}\big|x_{t},u_{t}^k, g_t^{-k}(\delta_t, \lambda_t^{-k})) \nonumber \\
& . \prod_{j=1, j \neq k}^K { \mu}_{\gamma_t^{j}(\delta_t, \lambda_t^j)}(u_t^{j}) \; \xi_{t}^{k} [\delta_t, \lambda_t^k](x_{t}, \lambda_{t}^{-k})dx_t \label{apost_1a} 
\end{align}
$ { \mu}_{\gamma_t^{j}(\delta_t, \lambda_t^j)}(u_t^{j})$  is the delta function at $u_t^j=\gamma_t^{j}(\delta_t, \lambda_t^j)$.   \\
For  $n=1$,  (\ref{apost_1})-(\ref{apost_1a}) hold with  the term $\prod_{j=1, j \neq k}^K{ \mu}_{\gamma_t^{j}(\delta_t, \lambda_t^j)}(u_t^{j}) $ removed  (i.e.,  $\lambda_{t+1}^{-k}=Y_{t+1}^{-k}$). \\
Moreover, the following hold. 

(1) For each $k$, 
 $\xi_{t}^k[\delta_t, \lambda_{t}^k]\equiv  {\bf P}_{t}^{{g^k, g^{-k}}} (x_{t}, \lambda_{t}^{-k}\big|\Delta_{t}=\delta_t, \Lambda_t=\lambda_t^k),\forall  t$ depends on the actions $u_{0,t-1}^k\in {\mathbb U}_{0,t-1}^k\tri \times_{j=1}^{t-1} {\mathbb U}_j^k$ and not   the strategies  $ g^k(\cdot)\in {\bf  U}_{0,t-1}^k$,   i.e., 
\begin{align}
{\bf P}_{t}^{{g^k, g^{-k}}} (x_{t}, \lambda_{t}^{-k}\big|\delta_{t}, \lambda_t^k)={\bf P}_{t}^{{g^{-k}}} (x_{t}, \lambda_{t}^{-k}\big|\delta_{t}, \lambda_t^k), \; \forall (t,   g^k).\nonumber 
\end{align}
\indent (2) 
 For each $k$,  the process  $\{ \Xi_t^{k}[\Delta_t, \Lambda_t^k]| t=0, \ldots, T\}$   is  conditionally  Markov w.r.t. $\{ \Delta_t | t=0, \ldots, T\}$, i.e., 
 \begin{align}
&{\bf   P}_{t+1}^{g^{k}, g^{-k}}(\xi_{t+1}^k \Big| \delta_{0,t}, \lambda_{0,t}^k, u_{0,t}^k) \label{ext-m}   \\
&={\bf P}_{t+1}^{g^{-k}}(\xi_{t+1}^k \Big| \xi_t^k, \delta_t,u_t^k), \hso  \forall g^{-k}, \; t=0,1, \ldots, T. \nonumber
\end{align}
\end{theorem} 
\begin{proof}  The derivation is found in \cite{CDC-GU-SD:CDC2026-ArXiv}. 
\end{proof}

\ \

\begin{remark} The properties of Theorem~\ref{thm:is-pbp}.(1), (2)  are essential  in the  development of the   DP  for PbP optimality. 

(I) The major difference compared to  the posterior centralized filtering PDF  of classical centralized POMDPs   \cite{bertsekas2005,kumar-varayia:B1986},  is that,   the operator ${\bf  T}_{k+1}^k\big(\cdot\big)$ in addition to  $\{\Pi_t^k, Y_{t+1}^k, u_t^k\}$, also depends on   $\Delta_t$ through     $g_t^{-k}(\Delta_t, \cdot)$.  \\ If  $K=1$ then there  is only one strategy, $g^1$, $\{\delta_t, \lambda_t^1\}=\{y_{0,t}^1, u_{0,t-1}^1\}$,   $\lambda_t^{-1}=\{\emptyset\}, \forall t$,  and ${\bf P}_{t}^{{g^k, g^{-k}}} (x_{t}, \lambda_{t}^{-k}\big|\delta_{t}, \lambda_{t}^k)\big|_{K=1}={\bf P}_{t}(x_{t} \big|y_{0,t}^1, u_{0,t-1}^1)$  satisfying 
 the recursion (\ref{apost_1}) reduces to the nonlinear filtering recursion     (\ref{Mal_49-c}),      as expected \cite{bertsekas2005,kumar-varayia:B1986}.

(II) An important property of private   \^a posteriori PDF $\{\Xi_t^k[\Delta_t, \Lambda_t^k],  t=0, \ldots, T\}$ is that it depends on the actions $u_t^k\in {\mathbb U}_t^k, t=0, \ldots, T-1$ and not on the strategies $g^k(\cdot) \in {\bf U}_{0,T-1}^k$, precisely as the  posterior centralized filtering PDF. 
\end{remark} 

\ \

Theorems~\ref{thm:vp-mn} presents the  DP equations  based on PbP optimality. These  illustrate  their explicit dependence on   the private \^a posteriori PDFs $\{\Xi_t^k[\Delta_t, \Lambda_t^k],  t=0, \ldots, T\}$  and the common information component $\{\Delta_t,  t=0, \ldots, T\}$. 

\ \

\begin{theorem}(Decentralized DP Equations  for PbP Optimality-Necessary Conditions and Verification Theorem)
\label{thm:vp-mn}
{\it (1) Necessary Conditions.}  For each $ k$,  suppose a PbP optimal $g^{k,o} \in {\bf  U}_{t,T-1}^k, \forall k$ exists.  

(1.1) 
The value processes of Definition~\ref{def:ctg-nm}  satisfy, $\forall (\delta_t, \lambda_t^k)$, 
\begin{align}
 {\cal V}_{t}^{g^{k,o},g^{-k,o}}(\delta_t,\lambda_t^k) =V_t^{g^{-k,o}}(\xi_t^k,\delta_t,\lambda_t^k ), \;  \forall (t,  k)   
\end{align}
where 
 $ { V}_t^{ g^{-k,o}}(\cdot),  
 \forall ( \xi_t^k, \delta_t, \lambda_t^k), \forall t$ are solutions of the following recursions:
\begin{align}
&{ V}_T^{g^{-k,o}}(\xi_T^k,\delta_T, \lambda_T^k  )={\mathbb E}^{ g^{-k,o}}   \Big\{ \label{dp_1_nnn-mn}   \\
& c_T(X_T)   \Big| \xi_T^k,\delta_T, \lambda_T^k \Big\}, \hst \forall k \in {\mathbb Z}_+^K \nonumber\\
&= \int_{{\mathbb X}_T  \times {\mathbb L}_{T}^{-k} }     c_T(x_T) \xi_{T}^{k}[\delta_T, \lambda_T^k] (x_{T}, \lambda_T^{-k})dx_T d\lambda_T^{-k}, \label{dp_1_nnn_1-mn}\\
&{ V}_t^{\gamma^{-k,o}}(\xi_t^k,\delta_t,\lambda_t^k  )  = \inf_{u_t^k \in {\mathbb U}_t^k}  {\mathbb E}^{ g^{-k,o}}   \Big\{ \label{dp_2_mn} \\
&   c_t(X_t,u_t^k,  g_t^{-k,o}(\delta_t, \Lambda_t^{-k}))  +{ V}_{t+1}^{g^{-k,o}}(\Xi_{t+1}^k, \Delta_{t+1}, \Lambda_{t+1}^k  )\nonumber\\
 & \Big| \xi_t^k,\delta_t, \lambda_t^k , U_t^k \Big\}, \hso   t=0,\ldots,T-1, \;   \forall k  \label{dp_2_mn}  \\
 &=  \inf_{u_t^k \in {\mathbb U}_t^k}  \Big\{  \int_{{\mathbb X}_t \times {\mathbb L}_{t}^{-k}}    c_t(x_t,u_t^k,  g_t^{-k,o}(\delta_t, \lambda_t^{-k}))  \nonumber \\
 &.  \xi_{t}^k[\delta_t, \lambda_t^k] (x_{t}, \lambda_{t}^{-k})dx_t d\lambda_t^{-k}\nonumber 
   \end{align}
 \begin{align}
 &  +  \int_{ {\mathbb Y}_{t+1}^k \times {\mathbb X}_{t, t+1}   \times {\mathbb L}_{t}^{-k}    } {V}_{t+1}^{g^{-k,o}}({\bf T}_{t+1}^{k}\big(y_{t+1}^{k}, u_t^k, g_t^{-k}(\delta_t, \cdot),    \xi_t^k(\cdot)\big), \nonumber  \\
 & \delta_t, \lambda_t^k ,   y_{t+1}^k, u_t^k,    y_{t-n+1}^{-k}, u_{t-n+1}^{-k})\label{dp_1_na_m}   \\
 &. {\bf P}_{t+1}^{ g^{-k,o}}(y_{t+1}^k, \lambda_{t}^{-k}, x_{t}, x_{t+1}\Big|   \xi_t^k, \delta_t,  u_t^k) dy_{t+1}^k  d\lambda_{t}^{-k} dx_{t} dx_{t+1} \Big\}  \nonumber
\end{align}
where   
 the joint PDF  of the last RHS term is
%
\begin{align}
& 
 {\bf P}_{t+1}^{{ g^{-k,o}}}(y_{t+1}^k, \lambda_{t}^{-k},x_t, x_{t+1} \big|  \xi_t^k, \delta_t, u_t^k)  \label{Mar_REC_1-mn-aaa}  \\
&= Q_{t+1}^k(y_{t+1}^k\big|    x_{t+1}){ S}_{t+1}(x_{t+1}\big|   x_t, u_t^k,g_t^{-k,o}(\delta_t, \lambda_t^{-k}))\nonumber \\
&.\xi_{t}^k[\delta_t,\lambda_t^k](x_{t},\lambda_{t}^{-k})
\label{Mar_REC_1-mn}  
\end{align}
and $u_{t-n+1}^{-k}=g_{t-n+1}^{-k,o}=\{g_{t-n+1}^{j,o}(\delta_{t-n+1}, \lambda_{t-n+1}^{j})\}_{j=1, j\neq k}^K$. 

(1.2)  For each $k$, if  the infimum in the DP equations of part (1.1) exists, then the optimal strategy occurs in the set of semi-separated strategies defined by,  $u_t^{k,o}=\gamma_t^{k,o}(\xi_t^k,\delta_t,\lambda_t^k )\in {\mathbb U}_t^k, t=0, \ldots,  T-1$,  i.e., ${\cal G}_t^k\tri \{\Xi_t^k[\Delta_t, \Lambda_t^k],\Delta_t,\Lambda_t^k \}$ is a sufficient statistic for $U_t^k$, $\forall t$. 

{\it (2) Sufficient Conditions.} Suppose the value functions $V_t^{g^{-k,o}}(\cdot), \forall t,  \forall k $ 
satisfy the DP  recursions of part (1).\\
The inequalities hold a.s.  (almost surely), 
\begin{align}
&{V}_T^{g^{-k,o}}(\Xi_T^k,\Delta_T, \Lambda_T^k ) \label{in_1-a} \\ 
&\leq J_{T,T}^{g^{k}, g^{-k,o}}(\Delta_T, \Lambda_T^k) , \;   \forall g^k \in {\bf U}_{0,T}^k,\;   \forall k,\nonumber  \\
&{ V}_t^{g^{-k,o}}(\Xi_t^k,\Delta_t,\Lambda_t^k )  \label{in_1}\\
& \leq J_{t,T}^{g^{k}, \gamma^{-k,o}}(\Delta_t, \Lambda_t^k) , \;  t=0, \ldots, T-1, \; \forall g^k \in {\bf  U}_{0,T-1}^k, \; \forall k . \nonumber 
\end{align}
where ${J}_{t,T}^{g^{k},g^{-k,o}}(\Delta_t, \Lambda_t^k)$ is given by  (\ref{opt-ctg-pbp-tc}).

\end{theorem}
\begin{proof} The derivation is found in \cite{CDC-GU-SD:CDC2026-ArXiv}. 
\end{proof}

\ \


\begin{remark} (On Theorem~\ref{thm:vp-mn})\\
\indent  In Theorems~\ref{thm:vp-mn},   the optimization in the right hand side of the DP equations    is over the action spaces, $u^k \in {\mathbb U}_t^k$, and     ${\cal G}_t^k\tri \{\Xi_t^k,\Delta_t^{(K)},\Lambda_t^k \}$ is a sufficient statistic for
$g^{k,o}\in {\bf  U}_{0,T-1}^k$,  but  expansive w.r.t.  $\Delta_t,\forall  t$. To  address this issue we need to identify another   state $\big\{ \Theta_{t}[\Delta_t]  | t =0, \ldots, T\big\}$, i.e., which satisfies a Markov recursion. This direction is pursued in \cite{CDC-GU-SD:CDC2026-ArXiv}, and leads to two additional DP equations. 
\end{remark}

\section{Gaps in the Proof of \cite[Theorems~5]{malikopoulosIEEEAC2023}}
\label{section:gaps}
The main technical gap in the proof of Theorem 5
originates from the use of the conditional
independence condition  (\ref{d-6}), which does not hold. We show  this below. 

Consider \cite[Appendix D, Proof of Theorem~5]{malikopoulosIEEEAC2023}. 

We recall\footnote{Below we follow the notation in \cite{malikopoulosIEEEAC2023}.}   \cite[eqn(112)]{malikopoulosIEEEAC2023}, i.e., (\ref{Mal_112}):
\begin{align}
&p^{{\bf g}}(X{t+1} \big| \Delta_{t+1}, \Lambda_{t+1}^k)     \label{Mal_112}\\
&=\frac{ p(Y_{t+1}^k\big|X_{t+1})     p^{{\bf g}}(X_{t+1} \big| \Delta_{t+1}, \Lambda_{t}^k, U_t^k)}{\int_{{\mathbb X}_{t+1}}  p(Y_{t+1}^k\big|X_{t+1})     p^{{\bf g}}(X_{t+1} \big| \Delta_{t+1}, \Lambda_{t}^k, U_t^k) dx_{t+1} }. \hst \forall t.\nonumber 
\end{align}
We also recall the statement leading to  \cite[ eqn(114)]{malikopoulosIEEEAC2023}, i.e., (\ref{Mal_114}):

{\it ``By Lemma~6, $p^{{\bf g}}(X_{t+1} \big| \Delta_{t+1}, \Lambda_{t}^k, U_t^k)$ depends only on the control strategy\footnote{To be precise, the  statement about dependence of the probabilities  only on   ${\bf g}^{-k}$, is not obvious at this stage. The standard method to show it (see  \cite{kumar-varayia:B1986}) is to first   derive the recursion  \cite[eqn(49)]{malikopoulosIEEEAC2023}, i.e., (\ref{Mal_49}),  and then  apply  induction using  the recursion  to show this statement.  Nevertheless, we ignore this technicality.} ${\bf g}^{-k}$, so we can write $p^{{\bf g}^{-k}}(X_{t+1} \big| \Delta_{t+1}, \Lambda_{t}^k, U_t^k)$. Next''}
\begin{align}
&p^{{\bf g}^{-k}}(X_{t+1} \big| \Delta_{t+1}, \Lambda_{t}^k, U_t^k)  \nonumber \\
&=\int_{{\mathbb  X}_t }p^{{\bf g}^{-k}}(X_{t+1} \big| X_t, \Delta_{t+1}, \Lambda_{t}^k, U_t^k) \nonumber \\
& .  p^{{\bf g}^{-k}}(X_{t} \big| \Delta_{t+1}, \Lambda_{t}^k, U_t^k)dX_{t}, \hst \forall t. \label{Mal_114}\
\end{align}
We will show  the following equalities which are used to derive the recursion, and  stated  in \cite[eqn(115)]{malikopoulosIEEEAC2023}, {\bf i.e., (\ref{Mal_115}),  is  incorrect}:
\begin{align}
&p^{{\bf g}^{-k}}(X_t \big| \Delta_{t+1}, \Lambda_t^k, U_t^k)= p^{{\bf g}^{-k}}(X_t \big| \Delta_{t}, Y_{t-n+1}^{1:K}, U_{t-n+1}^{1:K}, \nonumber \\
&\Lambda_t^k) = p^{{\bf g}^{-k}}(X_t \big| \Delta_{t},Y_{t-n+1}^{-k}, U_{t-n+1}^{-k}, \Lambda_t^k) \label{Mal_115a} \\
&{\bf \sr{?}{=}} p^{{\bf g}^{-k}}(X_t \big| \Delta_{t}, \Lambda_t^k). p^{{\bf g}^{-k}}(Y_{t-n+1}^{-k}, U_{t-n+1}^{-k}\big| \Delta_t,   \Lambda_t^k) \label{Mal_115}
\end{align}
{\it ``where $Y_{t-n+1}^{-k}=(Y_{t-n+1}^1, \ldots, Y_{t-n+1}^{k-1},Y_{t-n+1}^{k+1}, \ldots, Y_{t-n+1}^{K} )$, and $U_{t-n+1}^{-k}=(U_{t-n+1}^1, \ldots, U_{t-n+1}^{k-1},U_{t-n+1}^{k+1}, \ldots, Y_{t-n+1}^{K} )$. In the second equality, we dropped $Y_{t-n+1}^k, U_{t-n+1}^k$ from conditioning since they both  are included in $\Lambda_t^k$.''}

{\bf Next, we show the identity $\mbox{  (\ref{Mal_115a}) = (\ref{Mal_115})}$   is incorrect. } By Bayes' theorem, we have 
\begin{align}
\mbox{(\ref{Mal_115a})}
= \frac{  p^{{\bf g}^{-k}}(X_t,  \Delta_{t},Y_{t-n+1}^{-k}, U_{t-n+1}^{-k}, \Lambda_t^k)}{ p^{{\bf g}^{-k}}( \Delta_{t},Y_{t-n+1}^{-k}, U_{t-n+1}^{-k}, \Lambda_t^k)}\label{d-2}  
\end{align}
where the numerator is 
\begin{align}
&p^{{\bf g}^{-k}}(X_t,  \Delta_{t},Y_{t-n+1}^{-k}, U_{t-n+1}^{-k}, \Lambda_t^k)  \label{d-3}  \\
&= p^{{\bf g}^{-k}}(Y_{t-n+1}^{-k}, U_{t-n+1}^{-k}\big| X_t, \Delta_t,   \Lambda_t^k) \nonumber \\
&. p^{{\bf g}^{-k}}(X_t \big| \Delta_{t}, \Lambda_t^k). p^{{\bf g}^{-k}}( \Delta_{t}, \Lambda_t^k). \label{d-4}  
\end{align}
Substituting (\ref{d-4}) into (\ref{d-2}) we have, 
\begin{align}
&\mbox{(\ref{Mal_115a})}\nonumber \\
&=\frac{p^{{\bf g}^{-k}}(Y_{t-n+1}^{-k}, U_{t-n+1}^{-k}\big| X_t, \Delta_t,   \Lambda_t^k). p^{{\bf g}^{-k}}(X_t \big| \Delta_{t}, \Lambda_t^k) }{\int_{{\mathbb  X}_t}   p^{{\bf g}^{-k}}(Y_{t-n+1}^{-k}, U_{t-n+1}^{-k}\big| X_t, \Delta_t,   \Lambda_t^k). p^{{\bf g}^{-k}}(X_t \big| \Delta_{t}, \Lambda_t^k)   }.  \label{d-5}  
\end{align}
From (\ref{d-5}) we deduce that identity $\mbox{  (\ref{Mal_115a}) = (\ref{Mal_115})}$ holds if and only if  the conditional independence condition holds: 
\begin{align}
& p^{{\bf g}^{-k}}(Y_{t-n+1}^{-k}, U_{t-n+1}^{-k} \big| X_t, \Delta_t,   \Lambda_t^k)  \nonumber \\
& =p^{{\bf g}^{-k}}(Y_{t-n+1}^{-k}, U_{t-n+1}^{-k}\big| \Delta_t,   \Lambda_t^k),  \hst \forall t, \; \forall k \label{d-6}   \\
&\mbox{Equivalently the Markov chain holds:}\nonumber \\
&X_t \leftrightarrow (\Delta_t,   \Lambda_t^k) \leftrightarrow Y_{t-n+1}^{-k}, U_{t-n+1}^{-k}, \hso \forall t, \: \forall k \label{d-7}  
\end{align} 
{\bf Therefore, it suffices to  show (\ref{d-6}) 
 or equivalently  (\ref{d-7}) 
  do not hold. } To this end, consider   for simplicity,  $K=2, n=1$. Then by  definition,  $\Delta_t=\{Y_{0, t-1}^1,U_{0, t-1}^1,  Y_{0, t-1}^2, U_{0, t-1}^2\}, \Lambda_t^1=\{ Y_{t}^1\},\Lambda_t^2=\{ Y_{t}^2\}$. We check whether  conditional independence  (\ref{d-6}) holds:  
\begin{align}
& p^{{\bf g}^{-k}}(Y_{t-n+1}^{-k}, U_{t-n+1}^{-k}\big| X_t,  \Delta_t,   \Lambda_t^k)\Big|_{K=2,k=1,  n=1}   \label{d-8}  \\
&=p^{{\bf g}^{2}}(Y_t^2, U_t^2\big| X_t,  \{Y_{0, t-1}^1,  U_{0, t-1}^1, Y_{0, t-1}^2, U_{0,t-1}^2\}, Y_{t}^1)  \label{d-9}   \\
& \sr{(?)}{=} p^{{\bf g}^{2}}(Y_t^2, U_t^2\big|  \{Y_{0, t-1}^1, U_{0,t-1}^1,  Y_{0, t-1}^2, U_{0,t-1}^2\}, Y_t^1). \label{d-10} 
\end{align} 
{\bf Clearly, equality (\ref{d-10}) does not hold,}  because  by  (\ref{NDM-4}),   $Y_{t}^2   = h_{t}^2(X_{t}, Z_{t}^2), \forall t$, i.e., in (\ref{d-9}), the  conditioning variables  $\Delta_t=\{Y_{0, t-1}^1,  U_{0, t-1}^1, Y_{0, t-1}^2, U_{0,t-1}^2\}, \Lambda_t^1=\{ Y_{t}^1\}$   do not specify $X_t$. 
Hence, the  conditional independence condition  (\ref{d-6})  does not hold, and therefore  {\bf identity $\mbox{  (\ref{Mal_115a}) = (\ref{Mal_115})}$ is incorrect} (for equality to hold it is necessary and sufficient that the  conditional independence holds).

\ \

\begin{remark} In view of the errors regarding  
\cite[Theorem~5]{malikopoulosIEEEAC2023},
then  Lemma~7, Corollary~1, Lemma~8, Theorem~6, and Theorem~7 in \cite{malikopoulosIEEEAC2023} are also affected because they make use of Theorem~5 in \cite{malikopoulosIEEEAC2023}.
\end{remark}

\section{Conclusion}
In this paper, we have identified several critical gaps in  the proofs of \cite[Theorems 5, 7]{malikopoulosIEEEAC2023}. These gaps undermine the validity of the results presented in
\cite[Section~IV.B]{malikopoulosIEEEAC2023}, particularly those concerning decentralized DP based on  PbP optimality, and their consequences. We  have also presented the correct DP equations for PbP optimality, as recently developed in \cite{CDC-GU-SD:CDC2026-ArXiv,CDC:ECC2025}. The compression of $\{\Delta_t|t=0,1,\ldots, T\}$ via a Markov  information state,  (i.e., which  satisfies a Markov  recursion)  is developed in \cite{CDC-GU-SD:CDC2026-ArXiv}.

\bibliographystyle{IEEEtran}

\bibliography{Bibliography_Decentralized_arxiv_fixed}

\end{document}